\documentclass[a4paper,11pt]{amsart}
\usepackage{latexsym}
\usepackage{amssymb}
\usepackage{amsmath}

\usepackage{amsfonts}

\usepackage{amsmath,amssymb,amsthm,mathtools}
\usepackage{hyperref}
\hypersetup{%
  colorlinks,
  bookmarksopen,
  bookmarksnumbered,
  citecolor=blue,
  linkcolor=black,
  pdfstartview=FitH,
}

\newtheorem{theorem}{Theorem}[section]
\newtheorem{lemma}[theorem]{Lemma}
\newtheorem{proposition}[theorem]{Proposition}
\newtheorem{corollary}[theorem]{Corollary}
\theoremstyle{definition}
\newtheorem{definition}[theorem]{Definition}

\newtheorem{question}[theorem]{Question}

\theoremstyle{remark}
\newtheorem{remark}[theorem]{Remark}

\newcommand\nph{\varphi}

\DeclareMathOperator{\id}{id}

\DeclareMathOperator{\ran}{Ran}

\newcommand{\deq}{\stackrel{\text{\rm def}}{=}}

\newcommand{\cl}[1]{\mathcal{#1}}
\newcommand{\bb}[1]{\mathbb{#1}}

\newcommand{\sca}[1]{\left\langle#1\right\rangle} %

\def\dc#1{\expandafter\def\csname#1\endcsname{\mathcal{#1}}}
\def\db#1{\expandafter\def\csname b#1\endcsname{\mathbb{#1}}}
\def\df#1{\expandafter\def\csname f#1\endcsname{\mathfrak{#1}}}

\def\loopy#1#2{%
  \def#1##1{\def\next{#2{##1}#1}\ifx##1\relax\let\next\relax\fi\next}}

\loopy{\makemathcals}{\dc}
\loopy{\makemathbbs}{\db}
\loopy{\makemathfraks}{\df}

\let\epsilon\varepsilon
\let\phi\varphi
\makemathbbs DBCRNQZE\relax
\makemathcals QWERTYUIOPASDFGHJKLZXCVBNM\relax
\makemathfraks QWERTYUIOPASDFGHJKLZXCVBNM\relax
\newcommand{\rtext}[1]{\quad\text{#1}}

\makeatletter\newcount\@DT@modctr\newcount\@dtctr
\def\@modulo#1#2{\@DT@modctr=#1\relax\divide \@DT@modctr by #2\relax
\multiply \@DT@modctr by #2\relax\advance #1 by -\@DT@modctr}
\newcommand{\xxivtime}{\@dtctr=\time\divide\@dtctr by 60
\ifnum\@dtctr<10 0\fi\the\@dtctr:\@dtctr=\time\@modulo{\@dtctr}{60}%
\ifnum\@dtctr<10 0\fi\the\@dtctr}

\makeatother

\begin{document}

\title
{Schur idempotents and hyperreflexivity}
\author{G. K. Eleftherakis, R. H. Levene and I. G. Todorov}
\renewcommand{\shortauthors}{Eleftherakis, Levene and Todorov}

\address{Department of Mathematics, Faculty of Sciences, University of Patras\\
265~00 Patras, Greece}
\email{gelefth@math.upatras.gr}

\address{School of Mathematical Sciences, University College Dublin\\ Belfield, Dublin 4, Ireland}
\email{rupert.levene@ucd.ie}

\address{Pure Mathematics Research Centre, Queen's University Belfast\\ Belfast BT7 1NN, United Kingdom}
\email{i.todorov@qub.ac.uk}

\date{14 July 2015}

\begin{abstract}
  We show that the set of Schur idempotents with hyperreflexive range is a Boolean lattice 
  which contains all contractions. We establish a preservation result for sums which implies 
  that the weak* closed span of a hyperreflexive and a ternary masa-bimodule is hyperreflexive, 
  and prove that the weak* closed span of finitely many tensor products of a hyperreflexive space 
  and a hyperreflexive range of a Schur idempotent (respectively, a ternary masa-bimodule) is hyperreflexive.
\end{abstract}

\maketitle

\section{Introduction}\label{s_intro}

Arveson's distance formula \cite{arv} has played a fundamental role 
in operator algebra theory since its discovery, inspiring a great deal 
of research in several distinct settings (see \cite{dav} and \cite{dale} and 
the references therein). 
First established for nest algebras \cite{a_nest}, it is an estimate for the distance 
of an operator $T$ to an operator algebra $\cl A$ in terms of the
norms of the compressions of $T$ to suitable \lq\lq corners'' 
arising from the invariant subspace lattice of $\cl A$. 
A minimax property, the distance formula is not easily verified in practice
due to, firstly, the difficulty of computing specific operator norms, and 
secondly, the lack of knowledge of the invariant subspaces of a general $\cl A$. 
It however implies that $\cl A$ is a \emph{reflexive} operator algebra (see \cite{arv} and \cite{ls}); 
the presence of a distance formula for $\cl A$ is hence known as the \emph{hyperreflexivity}
of $\cl A$.

Arveson recognised the importance of maximal abelian selfadjoint algebras (masas, for short)
in the study of non-selfadjoint (non-abelian) operator algebras
\cite{a} and pioneered the use of masa-bimodules in operator algebra theory.
These are precisely the weak* closed invariant subspaces of 
weak* continuous masa-bimodule maps, also known as Schur multipliers --
a class of transformations that has played a central role in operator space theory 
since  Haagerup's characterisation \cite{haag}. 
In \cite{eletod}, we studied connections between Schur idempotents 
and reflexive masa-bimodules. In \cite{eletod2}, this study was extended 
by considering tensor products and their relation to 
operator synthesis. 
These papers showed that Schur idempotents are very instrumental 
in questions about reflexivity and related properties, and can be 
particularly useful for establishing preservation results.

The present article focuses on the role of Schur idempotents 
in hyperreflexivity problems. 
After collecting necessary background and setting notation in Section \ref{s_prel}, 
in Section \ref{s_lhr} we show that the set of all Schur idempotents with hyperreflexive ranges
is a Boolean lattice. While we are not able
to determine whether every Schur idempotent $\Phi$ has hyperreflexive range, 
we show that, if $\Phi$ belongs to the 
Boolean lattice $\frak{C}$ generated by the set $\frak{I}_1$ of contractive Schur idempotents, 
then it does so. Our results can thus be viewed as a test 
for the well-known (open) problem of whether the Boolean lattice $\frak{C}$
coincides with the set of all Schur idempotents:
the existence of a Schur idempotent with non-hyperreflexive range 
would imply a negative answer to the latter question. 
As a corollary, we show that all 
Schur bounded patterns \cite{dd} give rise to hyperreflexive subspaces. 

In Section \ref{s_hs}, we examine the behaviour of hyperreflexivity 
with respect to linear spans. We show that the sum of a hyperreflexive masa-bimodule
and the hyperreflexive range of a Schur idempotent is hyperreflexive, 
and use this to obtain a general result about linear spans 
(Theorem \ref{1.7}) which implies that the weak* closed linear span 
of a hyperreflexive masa-bimodule and a ternary masa-bimodule is hyperreflexive. 
Ternary masa-bimodules are subspace versions of type I von Neumann algebras 
and, together with the (more general) ternary rings of operators, 
have been extensively studied (see, {\it e.g.},~\cite{blm}, \cite{eletod} and \cite{eletod2}).

In Section \ref{s_inters}, we obtain results, analogous to the ones from 
Section \ref{s_hs}, but for intersections as opposed to linear spans. 
In particular, we prove that the intersection of an arbitrary hyperreflexive 
masa-bimodule and a subspace belonging to a general class, containing all ternary masa-bimodules,
is hyperreflexive. 

In Section \ref{s_htp}, we show that (finite, weak* closed) linear spans, each of whose term 
is the tensor product of a hyperreflexive space and a ternary masa-bimodule, is, 
under some natural condition, necessarily 
hyperreflexive (Theorem \ref{888}). 
This is achieved by showing first that a similar result holds in the case where 
the ternary masa-bimodules are replaced by hyperreflexive ranges of Schur idempotents. 

We wish to note that the results below are formulated for subspaces 
of operators acting on a single Hilbert space,
but they hold more generally for subspaces of operators between different spaces;
we have chosen to work on one Hilbert space in order to avoid somewhat 
cumbersome formulations. 

\section{Preliminaries}\label{s_prel}

Throughout this paper, we fix a separable Hilbert space $H$
and let $\B(H)$ denote the space of all bounded linear
operators on $H$. The norm on $H$ and the uniform 
operator norm on $\cl B(H)$ will both be denoted by $\|\cdot\|$.
Let $\X$ be a subspace of $\B(H)$. If $T\in \B(H)$, then the
\emph{distance} of $T$ to $\X$ is
\[ d(T, \X ) =\inf_{ X\in \X }\|T-X\| \] and the \emph{Arveson
  distance} of~$T$ to~$\X $ is
\[ 
  \alpha(T,\X ) =
  \sup\left\{\inf_{ X\in \X}\|T\xi -X\xi\| : \xi\in H, \|\xi\|=1\right\}. 
\]
Trivially, $\alpha(T,\X)\leq d(T,\X)$, and both~$d$ and~$\alpha$ are
order-reversing in the second variable. We say that $\X $
is~\emph{reflexive}~\cite{ls} if, whenever $T\in \B(H)$ is such that 
$T\xi\in \overline{\cl X\xi}$ for all $\xi\in H$, then $T\in \cl X$.  
Reflexive spaces are necessarily closed in the weak operator topology, 
and a weak* closed subspace $\cl X$
is reflexive precisely when 
\[ \alpha(T,\X ) = 0\implies d(T,\X) = 0, \ \ \ T\in \cl B(H).\] 
If~$\X$ satisfies the stronger
condition that there exist $k>0$ with
\begin{equation}\label{eq_ck}
d(T, \X ) \leq k\, \alpha(T,\X ), \ \ \  T\in \B(H),
\end{equation}
then~$\X$ is said to be \emph{hyperreflexive}; in this case, the least
constant $k$ for which (\ref{eq_ck}) holds is denoted by $k(\X)$ and
called the \emph{hyperreflexivity constant} of $\cl X$.  The space
$\X$ is called \emph{completely hyperreflexive} if $\X
\bar\otimes\B(\H )$ is hyperreflexive, where here and in the sequel
$\H$ is a separable infinite dimensional Hilbert space and
$\bar\otimes$ denotes the spatial weak* closed tensor product.  The
\emph{complete hyperreflexivity constant}~$k_c(\X)$ of $\X$ is by
definition the hyperreflexivity constant of $\X\bar\otimes\B(\H)$.  We
remark in passing that whether every hyperreflexive space is
completely hyperreflexive remains an open question~\cite{dale}.

\smallskip

We fix throughout a maximal abelian selfadjoint algebra (for short, masa) $\cl D$ on $H$. 
We denote by $\cl P(\cl D)$ the set of all projections in $\cl D$. 
A \emph{Schur multiplier} is a weak* continuous $\cl D$-bimodule map on $\cl B(H)$. 
The set of all Schur multipliers is a commutative algebra under pointwise addition and composition. 
If $\Phi$ is a Schur multiplier, we write $\|\Phi\|$
for the norm of~$\Phi$ as a linear map on the Banach space $\cl B(H)$.

A \emph{Schur idempotent} is a Schur multiplier $\Phi$ that is also an idempotent.
We denote by $\fI$ the collection of all Schur idempotents. It is easy to see that 
$\fI$ is a lattice under the operations $\Phi\wedge \Psi = \Phi \Psi$ and
$\Phi\vee \Psi = \Phi + \Psi - \Phi\Psi$, which is moreover Boolean for the 
complementation $\Phi\to \Phi^{\perp} \deq \id - \Phi$. 
For $\Phi,\Psi\in \fI$ we write $\Phi\leq \Psi$ if $\Phi\Psi = \Phi$, and
we denote by $\ran\Phi$ the range of $\Phi$.
We refer the reader to \cite{eletod} and \cite{levene} for more details on Schur idempotents. 

By a \emph{$\cl D$-bimodule} (or a \emph{masa-bimodule} when $\cl D$ is clear from the context) 
we mean a subspace $\cl X\subseteq \cl B(H)$ such that $\cl D\cl X\cl D\subseteq \cl X$. 
All  masa-bimodules in the sequel are assumed to be weak* closed. 
If $\Phi\in \fI$ then $\ran\Phi$ is easily seen to be a masa-bimodule. 

The statements in the next remark are straightforward. 

\begin{remark}\label{1.2} 
We have
\[  \alpha(T,\X) =\sup \{|\sca{T\xi, \eta}|: \;\;\|\xi \|=\|\eta\|=1, \; \sca{X\xi , \eta}=0,
\mbox{ for all } X\in \X\}. \]
Furthermore, if $\X$ is a $\D$-bimodule then  
\[ \alpha(T,\X) =\sup\big\{\|QTP\|:\;\;P,Q\;\in \P(\D), \;\; Q\X P=\{0\}\big\}. \]
\end{remark}

\section{The lattice of hyperreflexive ranges}\label{s_lhr}

In this section, we give a characterisation of the Schur idempotents with hyperreflexive
ranges and show that they form a sublattice of
the lattice $\frak{I}$ of all Schur idempotents.  We start by
formulating an alternative expression of the Arveson distance which
will prove useful in the sequel.

We write~$\fI_1=\{ \Phi\in\fI\colon \|\Phi\|\leq 1\}$ for the set of
contractive Schur idempotents. It was shown in \cite{kp} that a Schur idempotent $\Phi$
belongs to $\fI_1$ if and only if there exist families $(P_i)_{i\in \bb{N}}$
and $(Q_i)_{i\in \bb{N}}$ of mutually orthogonal projections in $\cl D$
such that 
\begin{equation}\label{eq_form1}
\Phi(T) = \sum_{i=1}^{\infty} Q_i T P_i, \ \ \ \ T\in \cl B(H),
\end{equation}
where the series converges in the weak* topology. 

\begin{proposition}\label{p_crr}
  Let $\X\subseteq \B(H)$ be a weak* closed $\cl D$-bimodule.  Then
  \[ 
  \alpha(T,\X) = 
  \sup\{\|\Phi(T)\| : \Phi\in \fI_1\text{ and } \Phi(\X) = \{0\}\}. 
  \]
\end{proposition}
\begin{proof}
  Let~$M$ be the supremum on the right hand side. By Remark~\ref{1.2},
  \[ 
  \alpha(T,\X) = \sup\{\|QTP\| : P,Q\in \cl P(\cl D) \mbox{ and }
    Q\X P = \{0\}\}.
  \] 
  Since any map of the form $T\mapsto QTP$ (where $P,Q\in \cl P(\cl D)$) 
  is in~$\fI_1$, we have $\alpha(T,\X)\leq M$.  
Conversely, suppose that $\Phi\in \fI_1$ and $\Phi(\X) = \{0\}$. 
Represent $\Phi$ as in (\ref{eq_form1}); then 
$Q_i \X P_i = \{0\}$ for each~$i$.  On the other hand, $\|\Phi(T)\| =
  \sup_{i\in \bN } \|Q_iTP_i\|\leq \alpha(T,\X)$, so $M\leq
  \alpha(T,\X)$.
\end{proof}

If~$\Phi\in \fI$, write
\[ \cl N_1(\Phi)=\{\Sigma\in \fI_1\colon \Sigma\Phi=0\}.\]
The following corollary is a direct consequence of Proposition \ref{p_crr}. 
  
\begin{corollary}\label{cor:alpha_range}
  If~$\Phi\in\fI$ and $T\in \B(H)$ then
  \[ \alpha(T,\ran \Phi)=\sup_{\Theta\in \cl N_1(\Phi)} \|\Theta(T)\|.\]
\end{corollary}

We next single out a simple condition that formally implies
hyperreflexivity. It is based on the fact that, 
if $\Phi$ is a Schur idempotent and $T\in \cl B(H)$, then
there is a canonical approximant of $T$ within $\ran\Phi$, namely the operator $\Phi(T)$. 

\begin{definition}\label{es} 
  We write $\fH$ for the set of Schur idempotents $\Phi \in\fI$ with
  the following property: there exists $\lambda > 0$ such that
  \[ \|\Phi^\perp (T)\|\leq \lambda\, \alpha(T,\ran\Phi), \ \ \ T\in
  \B(H).
  \]
  The least constant $\lambda$ with this property will be denoted by
  $\lambda(\Phi)$.
\end{definition}

If $\Phi\in \fI$ and $\ran\Phi$ is hyperreflexive, it will be convenient to 
denote by $k(\Phi)$ the hyperreflexivity constant $k(\ran\Phi)$. 

\begin{remark}\label{r_hh}
  Since $d(T,\ran\Phi)\leq \|T-\Phi(T)\|=\|\Phi^\perp(T)\|$, we see
  that if $\Phi\in \fH$, then $\ran\Phi$ is hyperreflexive and
  $k(\Phi)\leq \lambda(\Phi)$. We will show shortly that
  $\fH$ is precisely the set of Schur idempotents with
  hyperreflexive range.
\end{remark}

\begin{remark}\label{r_imm}
  In view of Proposition~\ref{p_crr}, if $\Phi$ is a Schur idempotent
  then $\Phi^{\perp}\in \fH$ precisely when there exists $\lambda > 0$ such that
$$\|\Phi(T)\|\leq  \lambda \sup\{\|\Theta(T)\| : \Theta\in \fI_1, \Theta\leq \Phi\}.$$
In particular, if $\Phi\in \fI_1$, then $\Phi^\perp\in\fH$ and $\lambda(\Phi^\perp)=1$.
\end{remark}

Recall that $\cl B(H)$ is the dual Banach space of the trace class
$\cl T(H)$ on $H$. 
Every element $\omega\in \cl T(H)$ 
is thus viewed as both an operator on $H$ and as a (weak* continuous) linear
functional on $\cl B(H)$; %
we denote by $\langle T,\omega\rangle$ the pairing 
between $T\in \cl B(H)$ and $\omega\in \cl T(H)$. 
If $f,g\in H$, we denote by $f\otimes g$ the rank one operator on $H$ given by 
$(f\otimes g)(\xi) = (\xi,g) f$, $\xi\in H$. 
We have that $\langle T, f\otimes g\rangle = (Tf,\overline{g})$, for 
a conjugate-linear isometry $g\to\overline{g}$ on $H$. 
If $\omega\in \cl T(H)$ then 
$\omega = \sum_{i=1}^{\infty} \omega_k$ in the trace norm $\|\cdot\|_1$, 
where $\omega_k$, $k\in \bb{N}$, are operators of rank one such that 
$\sum_{k=1}^{\infty} \|\omega_k\|_1 <\infty$. 

If $\cl X\subseteq \cl B(H)$, let 
$$\cl X_{\perp} = \{\omega\in \cl T(H) : \langle T,\omega\rangle = 0, \ \mbox{ for all } T\in \cl X\}$$
be the pre-annihilator of $\cl X$ in $\cl T(H)$.  The following result
was proved by Arveson \cite{arv} in the case the space $\cl X$ is a
unital algebra. The proof for the case where $\cl X$ is a subspace is
a straightforward modification of Arveson's arguments; this is also a
special case of~\cite[Theorem~2.2]{had}.

\begin{theorem}\label{arve}
Let $\X\subseteq \cl B(H)$ be a reflexive space. The following are equivalent: 

(i) \ $\cl X$ is hyperreflexive and $k(\cl X) \leq r$; 

(ii) for every $\omega\in\X_\bot$ and every $\epsilon>0$ there exists a sequence 
$(\omega_i)_{i\in \bb{N}}\subset\X_\bot$ of rank one operators such that 
\[\sum_{i=1}^{\infty}\|\omega_i\|_1 < (r + \epsilon) \|\omega\|_1 \ \mbox{ and } \ 
\omega = \sum_{i=1}^{\infty} \omega_i,\]
where the latter series converges in the trace norm. 
\end{theorem}

If $\Phi$ is a Schur idempotent, we write $\Phi_*$ for the predual of $\Phi$, acting 
on the trace class $\cl T(H)$. 

\begin{lemma}\label{l_prephi}
If $\Phi\in \fI$ and $\omega\in \cl T(H)$ then $\Phi^{\perp}_*(\omega)\in (\ran\Phi)_{\perp}$.
\end{lemma}
\begin{proof}
If $T\in \ran\Phi$ then 
\[\langle T,\Phi^{\perp}_*(\omega)\rangle = \langle \Phi(T),\Phi^{\perp}_*(\omega)\rangle
= \langle \Phi^{\perp}\Phi(T),\omega\rangle = 0.\qedhere\]
\end{proof}

\begin{theorem}\label{300} 
  Let $\Phi$ be a Schur idempotent. The following are equivalent:  
   
(i) \ $\ran\Phi$ is hyperreflexive;

(ii)  $\Phi\in\fH$.

Moreover, if these conditions hold then $\lambda(\Phi)\leq k(\Phi)\|\Phi^\perp\|$.
\end{theorem}
\begin{proof}  
(ii)$\Rightarrow$(i) was pointed out in Remark~\ref{r_hh}.

(i)$\Rightarrow$(ii)  
Let $k=k(\Phi)$ and fix $T\in\B(H)$.  For $\epsilon>0$ there exist unit vectors
  $\xi,\eta \in H$ with
\begin{equation}\label{eq_xieta}
  \|\Phi^\perp(T)\|-\epsilon<|(\Phi^\perp(T)\xi,\overline{\eta})|
  =|\langle\Phi^\perp(T),\xi\otimes\eta\rangle|=|\langle T, \Phi^\perp_*
  (\xi\otimes\eta)\rangle|. 
\end{equation}
By Lemma \ref{l_prephi}, $\Phi^\perp_*(\xi\otimes\eta)\in(\ran\Phi)_\bot$.
  Clearly,
  \[\|\Phi^\perp_* (\xi\otimes \eta)\|_1 \leq \|\Phi^\perp_*\|\,
  \|\xi\otimes \eta\|_1 =\|\Phi^\perp\|.\] 
  By Theorem~\ref{arve},
  there exist rank one operators $\omega_k\in (\ran\Phi)_\bot$, $k\in \bb{N}$, 
   such that
$$ \sum_{k=1}^\infty\|\omega_k\|_1
  \leq(k+\epsilon) \|\Phi^\bot\|  
  \ \ \mbox{ and } \ \ 
  \Phi^\perp_* (\xi\otimes\eta) = \sum_{k=1}^\infty \omega_k.$$ 
  By Remark~\ref{1.2} and~(\ref{eq_xieta}), 
\begin{eqnarray*}
  \|\Phi^\perp(T)\|-\epsilon
  & < & 
  \sum_{k=1}^\infty|\langle T, \omega_k \rangle| \leq \sum_{k=1}^\infty\|\omega_k\|_1 \alpha(T,\ran\Phi)\\
  & \leq & (k + \epsilon)\|\Phi^\perp\|\alpha(T,\ran \Phi).
\end{eqnarray*}
Since $\epsilon$ is arbitrary, 
$\|\Phi^\perp(T)\| \leq k \|\Phi^\bot\|\alpha(T,\ran\Phi)$. 
Thus, $\Phi\in \fH$ and $\lambda(\Phi)\leq k\|\Phi^\perp\|$.
\end{proof}

We next prove that the set~$\fH$ is closed under the lattice operations.

\begin{lemma}\label{lemma:complement}
  Let~$\Phi\in\fH$ and $\Sigma\in\fI_1$. Then the Schur idempotent
  $\Psi \deq (\Phi^\perp \Sigma)^\perp$
  belongs to $\fH$ and $\lambda(\Psi)\leq \lambda(\Phi).$
\end{lemma}
\begin{proof}
Note that $\Psi = \Sigma^\perp + \Phi\Sigma$. 
  Thus, if~$\Theta\in \cl N_1(\Phi)$, then $\Theta\Sigma\in
  \cl N_1(\Psi)$. Let $T\in \B(H)$; by Corollary \ref{cor:alpha_range}, 
  \begin{align*}
    \|\Psi^\perp(T)\|
    & = 
    \|\Phi^\perp(\Sigma(T))\| 
    \leq \lambda(\Phi)\,\alpha(\Sigma(T),\ran\Phi)\\
    & =  
    \lambda(\Phi)\,\sup_{\Theta\in \cl N_1(\Phi)} \|\Theta\Sigma(T)\| 
    \leq   \lambda(\Phi)\,\sup_{\Lambda\in \cl N_1(\Psi)} \|\Lambda(T)\|\\
    & =  
    \lambda(\Phi)\,\alpha(T,\ran\Psi).      \qedhere
  \end{align*}
\end{proof}

\begin{theorem}\label{est} 
  The set $\fH$ is a sublattice of~$\fI$.
\end{theorem}
\begin{proof}  
  Let $\Phi_1, \Phi_2 \in\fH$ and write~$\lambda_i=\lambda(\Phi_i)$
  and $\X_i=\ran \Phi_i$, $i = 1,2$. Set $\X \deq \X_1\cap\X_2=\ran(\Phi_1\Phi_2)$.
  For $T\in\B(H)$, we have
  \begin{align*} 
    \|T-\Phi_1 \Phi_2(T) \| &\leq \|T-\Phi_1(T)\|+\|\Phi_1(T) -\Phi_1 \Phi_2(T)\|
    \\&\leq \|T-\Phi_1(T)\|+\|\Phi_1\|\|T -\Phi_2(T)\|
    \\&\leq \lambda_1 \alpha(T,\X_1) + \lambda_2\|\Phi_1\|
    \alpha(T,\X_2).
  \end{align*} 
By the monotonicity of $\alpha$, we have   
  $\alpha(T,\X_i)\leq\alpha(T,\X)$, $i=1,2$.  Thus,
  \[ \|T-\Phi_1 \Phi_2(T) \| \leq(\lambda_1+ \lambda_2\|\Phi_1\|
  )\alpha(T,\X). \] 
  It follows that $\Phi_1\Phi_2\in \fH$.
  
  Now let $\W \deq \ran(\Phi_1\vee\Phi_2) = \X_1+\X_2$ and $T\in
  \B(H)$. Using Lemma~\ref{lemma:complement} and the fact that 
  $\W\subseteq\ran(\Sigma^\perp+\Phi_2\Sigma)$ for $\Sigma\in
  \cl N_1(\Phi_1)$, we have
  \begin{eqnarray*}
    \|(\Phi_1\vee\Phi_2)^\perp(T) \|
    & = & \|\Phi_1^\perp(\Phi_2^\perp(T))\| \leq \lambda_1\, \alpha(\Phi_2^\perp(T),\ran \Phi_1)\\
    & = & \lambda_1\, \sup_{\Sigma\in \cl N_1(\Phi_1)} \|\Phi_2^\perp \Sigma(T)\|\\
    & \leq & \lambda_1\lambda_2\, \sup_{\Sigma\in \cl N_1(\Phi_1)} \alpha(T,\ran(\Sigma^\perp+\Phi_2\Sigma))\\
    & \leq & \lambda_1\lambda_2\,\alpha(T,\W).
  \end{eqnarray*}
It follows that $\Phi_1\vee\Phi_2\in\fH$ and $\lambda(\Phi_1\vee\Phi_2)\leq \lambda_1\lambda_2$.
\end{proof}

Recall that a weak* closed masa-bimodule $\cl M$ is called \emph{ternary}, if it is a 
ternary ring of operators, that is, if $ST^*R\in \cl M$ whenever $S,T,R\in \cl M$
(see {\it e.g.}~\cite{blm}).

\begin{proposition}\label{1.3} 
If $\Phi$ is a contractive Schur idempotent then $\Phi\in \fH$ and $\lambda(\Phi)\leq 2$.
\end{proposition}
\begin{proof} 
By~\cite{kp}, the space $\M=\ran\Phi$ is a ternary masa bimodule. 
Consider the von Neumann algebra
  \[ \A= \left(\begin{array}{clr} [\M\M^*]^{-w^*} & \M\\ \M^* & [\M^*
      \M]^{-w^*}
\end{array}\right)
\subseteq \cl B(H\oplus H) \] and note that~$\D\oplus\D$ is a masa
in~$\B(H\oplus H)$, over which $\cl A$ is a bimodule.  
By~\cite[Lemma~8.3]{dav} there exists a contractive
idempotent $\cl D\oplus\cl D$-bimodule map
\[ \Psi : \cl B(H\oplus H)\rightarrow \cl B(H\oplus H) \] such that
$\A=\ran\Psi$ and 
\begin{equation}\label{ex1} 
  \|T-\Psi(T)\|  \leq 2\alpha(T,\A)\rtext{for all~$T\in \B(H\oplus H)$}.
 \end{equation}
 (Note that $\Psi$ is not necessarily a Schur idempotent on~$\B(H\oplus
 H)$ since it does not need to be weak*-continuous.)  Consider the isometry
\[ \theta\colon \B(H)\to \B(H\oplus H),\quad T\mapsto 
 \left(\begin{array}{clr} 0 & T \\ 0 & 0
\end{array}\right)\]
and observe that \[\Psi(\theta(T))=\theta( \Phi(T)),\rtext{for all $T\in
  \B(H)$}.\] Moreover, for~$T\in \B(H)$,
\begin{align*} 
  \alpha(\theta(T),\A) &= \sup_{\|\eta\|=\|\xi\|=1}\inf_{A\in\A}\|(\theta(T)-A)(\xi\oplus \eta)\| 
  \\& \leq
\sup_{\|\eta\|=\|\xi\|=1}\inf_{M\in\ran\Phi}\|(\theta(T)-\theta(M))(\xi\oplus \eta)\| \\&=
\sup_{\|\xi\|=1}\inf_{M\in\ran\Phi}\|(T- M)\xi\|=\alpha(T,\M). 
\end{align*}
By~(\ref{ex1}),
\[ \|\Phi^\perp(T)\|=\|T-\Phi(T)\|=\|\theta(T)-\Psi(\theta(T))\|\leq2
\alpha(\theta(T),\A)\leq 2\alpha(T,\ran \Phi).  \]
\end{proof}

  We write~$\fC = \fC(H)$ for the Boolean lattice generated by $\fI_1$ in $\fI$. 

\begin{corollary}\label{c_contains}
  $\fC\subseteq \fH$.
\end{corollary}
\begin{proof}
  Let $\fI_1^\perp =\{\Phi^\perp\colon \Phi\in \fI_1\}$. It is easy to see that the
  sublattice of~$\fI$ generated by~$\fI_1\cup \fI_1^\perp$ is Boolean 
  and hence it
coincides with $\fC$.  By Theorem~\ref{est}, Proposition \ref{1.3} 
  and Remark~\ref{r_imm}, $\fC\subseteq\fH$.
\end{proof}

\begin{question}
Does there exist a Schur idempotent whose range is not hyperreflexive?
In other words, is the second of the inclusions
\[ \fC\subseteq \fH\subseteq \fI\] 
strict?
If so, then this would imply that $\fC\ne \fI$, 
settling in the negative an open problem of several
  years' standing which asks whether the Boolean lattice generated by the contractive Schur idempotents
  exhausts all Schur idempotents.
\end{question}

We next show that a class of Schur idempotents,
studied by Varopolous~\cite{v} and by
Davidson and Donsig~\cite{dd}
(see also \cite{p_loc}), is contained in~$\fH$.  
Let $H = \ell^2$ and $\cl D$ be the masa of diagonal 
(with respect to the canonical basis) operators. 
A \emph{Schur bounded pattern}
\cite{dd} is a subset $\kappa\subseteq\bN\times\bN$ such that every
bounded function $\nph: \bN\times\bN\to\bC$ supported on $\kappa$ is a
Schur multiplier. If $\kappa$ is a Schur bounded pattern then the map $\Phi_{\kappa}$ of Schur
multiplication by the matrix $(a_{i,j})$, where $a_{i,j} = 1$ 
(resp. $a_{i,j} = 0$) if $(i,j)\in \kappa$ (resp. $(i,j)\not\in \kappa$) 
is a Schur idempotent. 

\begin{proposition}\label{p_sbp}
  If $\kappa\subseteq \bN\times\bN$ is a Schur bounded pattern, then
  $\Phi_\kappa\in \frak{H}$.
\end{proposition}
\begin{proof}
  By \cite{dd}, there exist sets $R,C\subseteq\bN\times\bN$ and 
  a constant $N\in\bN$ such that:
  \begin{enumerate}
  \item$\{j\in\bN: (i,j)\in R\}$ has at most $N$ elements for each $i\in\bN$;
  \item$\{i\in\bN: (i,j)\in C\}$ has at most $N$ elements for each $j\in\bN$; and
  \item$\kappa= R\cup C$.
  \end{enumerate}
  We have $\Phi_\kappa=\Phi_R\vee \Phi_C$. 
  By Theorem~\ref{est} it suffices to show that
  $\Phi_R\in \fH$ and $\Phi_C\in \fH$. It is easily seen, however, 
   that $\ran\Phi_R$ is the sum of at most $N$ ternary masa-bimodules,
  so it is hyperreflexive by Corollary~\ref{c_contains}. Hence $\Phi_R\in
  \fH$, and similarly, $\Phi_C\in \fH$.
\end{proof}

\section{Hyperreflexivity and spans}\label{s_hs}

In this section, we show that, under certain conditions, 
hyperreflexivity is preserved under summation. The main results of the section are 
Theorem \ref{1.7} and the subsequent Corollary \ref{c_ahit}. 
The first step is the following lemma. 

\begin{lemma}\label{1.5}
  If $\U$ is a hyperreflexive masa-bimodule and $\Phi\in\fH$, then the
  algebraic sum $\U+\ran\Phi$ is hyperreflexive and
\[k(\U+\ran\Phi)\leq  k(\U)\lambda(\Phi).\]
\end{lemma}
\begin{proof}
  By~\cite[Corollary 3.4]{eletod}, the algebraic sum 
  $\W=\U+\ran\Phi$ is weak* closed.
  Given projections~$P,Q\in \cl D$, let $\Sigma_{Q,P}$ be the (contractive) Schur idempotent
  given by $\Sigma_{Q,P}(T)=QTP$,
  and
  \[ \Psi_{Q,P} = \Sigma_{Q,P}^\perp + \Phi\Sigma_{Q,P} = (\Phi^\perp
  \Sigma_{Q,P})^\perp.\] 
Note that   
\[ \Psi_{Q,P}^\perp(T)= Q\Phi^\perp(T)P, \ \ \ T\in \cl B(H).\]
  By Lemma~\ref{lemma:complement}, $\Psi_{Q,P}\in \fH$ and 
  $\lambda(\Psi_{Q,P})\leq \lambda(\Phi)$.
Using Remark~\ref{1.2}, and writing $P,Q$ for projections in $\cl D$, we have
  \begin{eqnarray*}
    d(T,\U+\ran\Phi) 
    & = & \inf\{\|T - X - Y\| : X\in \cl U, Y\in \ran\Phi\}\\
    & \leq & \inf\{\|T - X - \Phi(T)\| : X\in \cl U\}\\ 
    & = & d(\Phi^\perp(T),\U)\\
    & \leq & k(\U)\,\alpha(\Phi^\perp(T),\U)\\
    & = & k(\U) \sup\{\|Q\Phi^\perp(T)P\| :  Q\U P=\{0\}\}\\
    & = & k(\U) \sup\{\|\Psi_{Q,P}^\perp(T)\| :  Q\U P=\{0\}\}\\
    & \leq & k(\U) \sup\{\lambda(\Psi_{Q,P})\, \alpha(T,\ran\Psi_{Q,P}) :  Q\U P=\{0\}\}\\
    & \leq & k(\U)\,\lambda(\Phi)\,\alpha(T,\U+\ran\Phi),
  \end{eqnarray*}
  since, if~$Q\U P=\{0\}$, then $\Psi_{Q,P}^\perp(\U+\ran\Phi)=\{0\}$,
  and hence $\U+\ran\Phi\subseteq \ran\Psi_{Q,P}$.
\end{proof}

\begin{corollary} 
  If\/ $\Phi\in \fI_1$ and\/ $\U$ is a hyperreflexive masa-bimodule, then
  the algebraic sum $\W=\U+\ran\Phi^{\perp}$ is hyperreflexive and
  $k(\W)\leq k(\U)$.
\end{corollary}
\begin{proof}
Immediate from Remark~\ref{r_imm} and Lemma \ref{1.5}.
\end{proof}

\begin{lemma}\label{1.6} 
  Let $\U_n$, $n\in \bb{N}$, be hyperreflexive spaces, such that 
$\cl U_{n+1}\subseteq \cl U_n$ for each $n\in \bb{N}$ and   
  $\sup_n k(\U_n)<\infty$. Then the space $\U \deq\bigcap_n\U_n$ is
  hyperreflexive and $k(\U)\leq\limsup_n k(\U_n)$.
\end{lemma}
\begin{proof} Let $T\in\B(H)$.  
  Since \[ d(T, \U_n)\leq k(\U_n)\,\alpha(T,\U_n), \] 
  there exists $S_n \in\U_n$ such that 
  \[
  \|T-S_n\|<k(\U_n)\alpha(T,\U_n) +\frac{1}{n}, \ \ \ n\in \bb{N}. \] 
 Since $\alpha(T,\U_n)\leq\alpha(T,\U)$, $n\in\bb{N}$, and $\sup_n k(\U_n)<\infty$, the
  sequence $(S_n)_{n\in \bb{N}}$ is bounded, and hence after passing to a subsequence
  we may assume that~$(S_n)$ converges in the weak* topology to some
  operator~$S$.  Writing $k = \limsup_n k(\U_n)$, we have
  \[ \|T-S\|\leq\limsup_n\|T-S_n\|\leq k \limsup_n\alpha(T,\U_n). \]
We thus conclude that 
\[ d(T,\U)\leq\|T-S\|\leq\limsup_n k(\U_n)\, \alpha(T,\U).\qedhere \]
\end{proof}

We next introduce a hyperreflexivity analogue of 
approximately $\fI$-injective masa-bimodules defined in \cite{eletod}.
  Let us say that a uniformly bounded sequence $(\Phi_n)_{n\in \bb{N}}\subseteq \fI$
  decreases to a subspace $\V\subseteq \cl B(H)$ 
  if $\Phi_1\geq \Phi_2\geq \dots$ and $\V=\bigcap_n
  \ran\Phi_n$. Recall~\cite{eletod} that in this case, the masa bimodule~$\V$ is
  said to be \emph{approximately $\fI$-injective}.

\begin{definition}
  A masa-bimodule $\V\subseteq \cl B(H)$ will be called \emph{approximately $H$-injective} if~there is
  a uniformly bounded sequence $(\Phi_n)_{n\in\bN}$ which decreases to~$\V$ such that
  \[\Phi_n\in\fH \text{ for each } n\in \bb{N} \mbox{ and } \sup_{n\in\bN} \lambda(\Phi_n) <
  \infty.\] The greatest lower bound of 
  the possible values of the latter supremum will be denoted by~$\lambda_H(\V)$.
\end{definition}

\begin{theorem}\label{1.7} 
  If $\V$ is an approximately H-injective masa-bimodule and $\U$ is a
  hyperreflexive masa-bimodule, then the algebraic sum $\U+\V$ is
  hyperreflexive and \[k(\U+\V)\leq k(\U)\,\lambda_H(\V).\]
\end{theorem}
\begin{proof} 
  Let $(\Phi_n)_{n\in\bN}$ be a uniformly bounded sequence in~$\fH$ decreasing to~$\V$
  with $\lambda \deq \sup_n \lambda(\Phi_n) < \infty$. 
  By Lemma~\ref{1.5},
  \[ k(\U+\ran \Phi_n)\leq k(\U)\,\lambda(\Phi_n),\] so $\sup_n
  k(\U+\ran \Phi_n)<\infty$. 
By~\cite[Corollary 3.4]{eletod}, the space 
$\U+\V$ is weak* closed, and by the proof of \cite[Theorem 2.5]{eletod}, 
  $$\U+\V=\U+\bigcap_n\ran \Phi_n=\bigcap_n (\U+\ran \Phi_n).$$ 
  By Lemma~\ref{1.6}, $\U+\V$ is hyperreflexive and
$$    k(\U+\V)
    =  k\Big(\bigcap_n(\U+\ran \Phi_n)\Big) \leq \limsup_n k(\U+\ran \Phi_n)
    \leq \lambda k(\U).$$
Taking an infimum over all possible values of $\lambda$, we obtain $k(\U+\V)\leq k(\U)\, \lambda_H(\V)$.
\end{proof}

\begin{corollary}\label{c_ahit}
  If $\U$ is a hyperreflexive masa-bimodule and $\M$ is a ternary
  masa-bimodule, then\/ $\U+\M$ is hyperreflexive and $k(\U+\M)\leq
  2k(\U)$.
\end{corollary}
\begin{proof}
  It is well-known that every ternary masa-bimodule is the intersection of 
  a descending sequence of ranges of contractive Schur idempotents
  (see, {\it e.g.},~\cite{eletod}). By Proposition~\ref{1.3}, 
  $\M$ is approximately $H$-injective and
  $\lambda_H(\M)\leq2$.  The statement now follows from Theorem
  \ref{1.7}.
\end{proof}

\section{Hyperreflexivity and intersections}\label{s_inters}

In this section, we show that the intersection of a hyperreflexive masa-bimodule 
and an approximately $H$-injective one is hyperreflexive. 
We first establish this statement in a special case. 

\begin{lemma}\label{xxx1} If $\cl U$ is a hyperreflexive masa-bimodule and $\Phi \in \fH$, 
then the intersection $ \cl U\cap \ran \Phi  $ 
is hyperreflexive and $$k( \cl U\cap \ran \Phi  )\leq \lambda (\Phi )+\|\Phi \|k(\cl U).$$
\end{lemma}
\begin{proof} 
Let $\cl W=\cl U\cap \ran \Phi$. Since $\cl U$ is invariant under $\Phi$, we have
\begin{equation}\label{eq_U}
\cl W = \{\Phi (X): X\in \cl U\}.
\end{equation}
For arbitrary $T\in \B(H)$ we have 
\begin{eqnarray*} 
\|T-\Phi (X)\| & \leq & \|T-\Phi (T)\|+ \| \Phi \| \|T-X\|  \\ 
& \leq & \lambda (\Phi )\alpha (T, \ran \Phi )+ \|\Phi \|\|T-X\| .  
\end{eqnarray*}
Thus, 
$$\inf_{X\in \cl U}\|T-\Phi (X)\|\leq \lambda (\Phi )\alpha (T, \ran \Phi )+ \|\Phi \|\inf_{X\in \cl U}\|T-X\| $$ 
and, by (\ref{eq_U}),
 \begin{eqnarray*} 
d(T, \cl W) 
& \leq & \lambda (\Phi )\alpha (T, \ran \Phi )+ \|\Phi \|d(T,\cl U)  \\
& \leq & \lambda (\Phi )\alpha (T, \ran \Phi )+ \|\Phi \|k(\cl U)\alpha (T, \cl U).
\end{eqnarray*}
By the monotonicty of $\alpha$, we have
\[d(T, \cl W)\leq (\lambda (\Phi )+\|\Phi \|k(\cl U))\alpha (T, \cl W).\qedhere\]
\end{proof}

\begin{theorem}\label{xxx2} 
If $\cl V$ is an approximately H-injective masa-bimodule and $\cl U$ is a hyperreflexive masa-bimodule, 
then the intersection $\W =\U\cap \V$ is hyperreflexive and 
$$k(\cl W)\leq \lambda _H(\V)+ k(\cl U) + \lambda _H(\cl V)k(\cl U).$$ 
\end{theorem}
\begin{proof} 
Let $(\Phi _n)_{n\in \bb N}$ be a uniformly bounded sequence in $\fH$ decreasing to $\V$ with 
$$\lambda =\sup_{n\in \bb{N}}\lambda (\Phi _n) < \infty.$$
Since $\|\Phi _n^\bot \|\leq \lambda (\Phi _n)$ for all $n,$ we have 
\begin{equation}\label{eq_1plu}
\sup_{n\in \bb{N}}\|\Phi _n\|\leq 1+\lambda .
\end{equation}
By the proof of \cite[Theorem 2.5]{eletod}, 
$$\W = \cap _{n=1}^\infty (\cl U\cap \ran \Phi _n).$$ 
By (\ref{eq_1plu}) and Lemma \ref{xxx1},
$$k(\cl U\cap \ran \Phi _n)\leq \lambda +(1+\lambda )k(\cl U), \ \ \ n\in \bb{N}.$$ 
Lemma \ref{1.6} now implies that $\cl W$ is hyperreflexive and 
$$k(\cl W)\leq \lambda +(1+\lambda )k(\cl U).$$
The stated estimate follows after taking the infimum over all possible values of~$\lambda$.
\end{proof}

Using Theorem \ref{xxx2} and arguing as in the proof of Corollary \ref{c_ahit}, 
we obtain the following corollary.

\begin{corollary}\label{xxx3} 
If $\cl U$ is a hyperreflexive masa-bimodule and $\M$ is a ternary masa-bimodule then
$\U\cap \M$ is hyperreflexive and 
$$k(\U\cap \M)\leq 2+3k(\U).$$
\end{corollary}

\begin{corollary}\label{xxx4} 
If $\cl U$ is a weak* closed nest algebra bimodule and $\M$ is a ternary masa-bimodule then 
$$k(\U\cap \M)\leq 5.$$
\end{corollary}
\begin{proof} 
The statement is immediate from Corollary \ref{xxx3} and the fact that $k(\U)=1$ \cite{dale}. 
\end{proof}

\section{Hyperreflexivity and tensor products}\label{s_htp}

In this section, we establish a preservation result for hyperreflexivity 
under the formation of tensor products. 
In addition to the Hilbert space $H$, 
we fix a separable Hilbert space $K$ and a masa in $\cl B(K)$. 
If $\cl U\subseteq \cl B(H)$ and 
$\cl V\subseteq \cl B(K)$ are subspaces, we denote by 
$\cl U\otimes\cl V$ their algebraic tensor product, viewed as a subspace of $\cl B(H\otimes K)$, 
so that $\cl U\bar\otimes\cl V$ is the weak* closure of $\cl U\otimes\cl V$.

Recall that $\fC(H)$ denotes the Boolean lattice generated by the
contractive Schur idempotents acting on $\cl
B(H)$. By~(\ref{eq_form1}), it is easy to see that if $\Phi$ is a
contractive Schur idempotent on $\B(H)$, then $\Phi\otimes\id$ is a
contractive Schur idempotent on~$\B(H\otimes K)$. Since tensoring with
the identity map on $K$ commutes with the lattice operations, it
follows that if $\Phi\in \fC(H)$ then $\Phi\otimes \id \in
\fC(H\otimes K)$.  By Corollary \ref{c_contains},
$(\ran\Phi)\bar\otimes \cl B(K) = \ran(\Phi\otimes\id)$ is
hyperreflexive, so $\ran\Phi$ is completely hyperreflexive.  We let
$k_c(\Phi) = k_c(\ran\Phi)$, and $\lambda_c(\Phi) =
\lambda(\Phi\otimes\id)$.

\begin{theorem}\label{777}
Let $\Phi_i\in\fC(H)$, $\cl X_i = \ran\Phi_i$, and 
$\U_i\subseteq \cl B(K)$ be a weak* closed subspace, $i=1,\dots,n$.
Suppose that, for every non-empty subset $E = \{i_1,\dots,i_m\}$ of the set $\{1,\dots,n\}$, the space 
$$\cl U_E \deq\overline{\U_{i_1} + \cdots + \cl U_{i_m}}^{w^*}$$ 
is completely hyperreflexive. 
Then the space
\[\W= \overline{\X_1\otimes\U_1+\dots+\X_n\otimes\U_n}^{w^*} \]  
is completely hyperreflexive.
\end{theorem}
\begin{proof}
It will be convenient to set $\cl U_{\emptyset} = \{0\}$. 
We first show that $\cl W$ is hyperreflexive. 
Let $\cl S$ be the set of all subsets of $\{1,\dots,n\}$. 
For $E\in \cl S$, let 
$$\Phi_E = \bigwedge_{i\in E} \Phi_i, \ \ \mbox{ and } \ \ \Psi_E = \bigvee_{i\in E} \Phi_i ,$$
where $\Phi_{\emptyset} = \id$ and $\Psi_{\emptyset} = 0$. 
Then
\begin{equation}\label{eq_phispsi}
\id = \sum_{E\in \cl S} \Phi_E \Psi_{E^c}^{\perp}.
\end{equation}
Note that, for each $E\in \cl S$, we have
\begin{equation}\label{eq_w}
\cl W  \subseteq \overline{\ran\Psi_{E^c}\otimes \B(K) + \cl B(H)\otimes \cl U_E}^{w^*}.
\end{equation}
Indeed, for each $i$, either $i\in E$, in which case 
$\cl X_i\otimes \cl U_i\subseteq \cl B(H)\otimes \cl U_E$, or $i\in E^c$, 
in which case $\cl X_i\otimes \cl U_i\subseteq \ran\Psi_{E^c}\otimes\B(K)$. 

For $E\in \cl S$, set 
$$\theta_E = (\Phi_E \Psi_{E^c}^{\perp})_* = (\Phi_E)_* (\Psi_{E^c}^{\perp})_*$$
and let $\omega\in \cl W_{\perp}$. 
Since $\cl W$ is invariant under the map $\Phi_E \Psi_{E^c}^{\perp}$, we have that 
\begin{equation}\label{eq_omin}
\theta_E(\omega) \in \cl W_{\perp}.
\end{equation}
By (\ref{eq_phispsi}), 
$\omega = \sum_{E\in \cl S} \theta_E(\omega)$.

We claim that 
\begin{equation}\label{eq_eom}
\theta_E(\omega) \in \big((\ran\Psi_{E^c})\otimes \cl B(K) + \cl B(H)\otimes \cl U_E\big)_{\perp}.
\end{equation}
To show (\ref{eq_eom}), suppose first that $X\in \ran\Psi_{E^c}$ and $B\in \cl B(K)$. 
Then 
$$\langle X\otimes B, \theta_E(\omega)\rangle
 =  \langle X\otimes B, (\Psi_{E^c}^{\perp})_*(\theta_E(\omega))\rangle
=  \langle \Psi_{E^c}^{\perp}(X\otimes B), \theta_E(\omega)\rangle = 0,$$
and hence 
\begin{equation}\label{eq_onep}
\theta_E(\omega) \in \big((\ran\Psi_{E^c})\bar\otimes \cl B(K)\big)_{\perp}.
\end{equation}
Now let $A\in \cl B(H)$ and $Y\in \cl U_E$. Then 
$$\Phi_E(A\otimes Y) \in (\cap_{i\in E}\cl X_i) \otimes \cl U_E\subseteq \cl W$$
and, using (\ref{eq_omin}), we see that 
$$\langle A\otimes Y, \theta_E(\omega)\rangle
=  \langle A\otimes Y, (\Phi_E)_*(\theta_E(\omega))\rangle
=  \langle \Phi_E(A\otimes Y), \theta_E(\omega)\rangle = 0.$$
Thus, 
\begin{equation}\label{eq_onep2}
\theta_E(\omega) \in (\cl B(H)\bar\otimes \cl U_E)_{\perp}.
\end{equation}
Now (\ref{eq_onep}) and (\ref{eq_onep2}) imply (\ref{eq_eom}). 

Let
$$k_E \deq k_c(\cl U_E)\prod_{i\in E^c} \lambda_c(\Phi_i)$$
and fix $\epsilon > 0$.
By Lemma \ref{1.5}, 
$(\ran\Psi_{E^c})\bar\otimes \cl B(K) + \cl B(H)\bar\otimes \cl U_E$ is hyperreflexive 
and 
$$k((\ran\Psi_{E^c})\bar\otimes \cl B(K) + \cl B(H)\bar\otimes \cl U_E)\leq 
k_c(\cl U_E)\lambda_c(\Psi_{E^c}) \leq k_E$$
since, by Lemma~\ref{1.5} and Remark~\ref{r_hh}, 
$$\lambda_c(\Psi_{E^c}) \leq \prod_{i\in E^c} \lambda_c(\Phi_i).$$
It thus follows from (\ref{eq_eom}) and Theorem~\ref{arve} that there exist rank one operators
$$\omega_E^\ell \in \big((\ran\Psi_{E^c})\bar\otimes \cl B(K) + \cl B(H)\bar\otimes \cl U_E\big)_{\perp},  \ \ \ell\in \bb{N},$$ 
such that 
$$\sum_{\ell=1}^{\infty} \|\omega_E^\ell\|_1 < (k_E +\epsilon)\|\theta_E(\omega)\| 
\leq (k_E +\epsilon) \|\Phi_E \Psi_{E^c}^{\perp}\| \|\omega\|_1$$
and 
$$\theta_E(\omega) = \sum_{\ell=1}^{\infty} \omega_E^\ell$$
in the trace norm.
It follows that 
$$\sum_{E\in \cl S} \sum_{\ell = 1}^{\infty} \|\omega_E^{\ell}\|_1 
\leq  \left(\sum_{E \in \cl S} (k_E +\epsilon) \|\Phi_E \Psi_{E^c}^{\perp}\|\right) \|\omega\|_1$$
and 
$$\omega = \sum_{E \in \cl S} \sum_{\ell = 1}^{\infty} \omega_E^{\ell}$$
in the trace norm.

Note that, by (\ref{eq_w}), $\omega_E^{\ell}\in \cl W_{\perp}$ for each $E\in \cl S$ and each $\ell\in \bb{N}$. 
By Theorem~\ref{arve}, $\cl W$ is hyperreflexive and 
\begin{equation}\label{eq_longes}
k(\cl W) \leq \sum_{E\in \cl S} k_c(\cl U_E) \prod_{i\in E^c} \lambda_c(\Phi_i) \|\Phi_E \Psi_{E^c}^{\perp}\|.
\end{equation}

To see that $\cl W$ is completely hyperreflexive, note that, if $\cl H$ is a separable Hilbert space, 
then 
$$\cl B(\cl H)\bar\otimes\cl W 
=  \overline{(\cl B(\cl H)\bar\otimes \X_1)\otimes\U_1+\cdots+(\cl B(\cl H)\bar\otimes \X_n)\otimes\U_n}^{w^*}.$$
Since $\cl B(\cl H)\bar\otimes \X_i = \ran(\id\otimes\Phi_i)$ and 
$\lambda_c(\id\otimes\Phi_i) = \lambda_c(\Phi_i)$, $i = 1,\dots,n$,
the claim now follows from the previous paragraphs.
\end{proof}

\begin{remark}
  It should be noted that the (complete) hyperreflexivity of the
  spaces $\cl U_E$ cannot be omitted from the assumptions of Theorem
  \ref{777}. Indeed, it is implied by its conclusion by taking $\Phi_i
  = \id$ for $i\in E$ and $\Phi_i = 0$ for $i\not\in E$.
\end{remark}

\begin{corollary}\label{444}
Let $\Phi\in\fC(H)$, $\cl X = \ran\Phi$ and $\U\subseteq\B(K)$ be a completely hyperreflexive subspace. 
Then the space $\W=\X\bar\otimes\U$ is 
hyperreflexive and \[ k(\W)\leq \lambda_c(\Phi)\|\Phi^{\perp}\| + k_c(\U)\|\Phi\| .\]
\end{corollary}
\begin{proof}
The claim is immediate from estimate (\ref{eq_longes}), after taking into account that 
$\lambda_c(\{0\}) = k_c(\{0\}) = 1$.
\end{proof}

\begin{corollary}\label{333} 
Let $\{\Phi_1,\dots,\Phi_n\} \subseteq\fC(H),$ and $\{\Psi_1,\dots,\Psi_n\}\subseteq\fC(K)$. 
If $\X_i=\ran\Phi_i$ and $\Y_i=\ran\Psi_i$, $i = 1,\dots,n$, 
then the space 
\[ \X_1\bar\otimes\Y_1+\dots+ \X_n\bar\otimes\Y_n \] is hyperreflexive.
\end{corollary}
\begin{proof}
The statement is immediate from Theorem \ref{777}, Theorem \ref{est} and the 
fact that $\cl W$ is weak* closed (see \cite[Corollary 3.4]{eletod}).
\end{proof}

\begin{remark}
  The preceding three results hold more generally (with identical
  proofs) if we replace $\fC(H)$ in the hypotheses with the 
  lattice $\fH_c$ of Schur idempotents with completely hyperreflexive
  range:
  \[ \fH_c = \{ \Phi\in \fH \colon \lambda_c(\Phi)<\infty\}.\] On the
  other hand, $\fC(H)$ seems a more natural class to work with.
\end{remark}

\begin{theorem}\label{888} 
Let $\M_i\subseteq \cl B(H)$ be a ternary masa-bimodule and 
$\U_i\subseteq \cl B(K)$ be a weak* closed subspaces, $i=1,\dots,n$.
Suppose that for every non-empty subset $E = \{i_1,\dots,i_m\}$
of the set $\{1,\dots,n\}$, the subspace 
$$\cl U_E \deq\overline{\U_{i_1} + \cdots + \cl U_{i_m}}^{w^*}$$ 
is completely hyperreflexive. 
Then the space  
\[\W= \overline{\M_1\otimes\U_1 + \cdots + \M_n\otimes\U_n}^{w^*} \]  
is completely hyperreflexive.
\end{theorem}
\begin{proof} 
  As in the proof of Corollary~\ref{c_ahit}, we may write
\[ \M_i=\bigcap_{j=1}^\infty\ran\Phi_j^i, \;\;\;i=1,\dots,n, \] where each $\Phi_j^i$ is a 
contractive Schur idempotent such that $\Phi_{j+1}^i\leq\Phi_j^i$ for all $i$ and $j$. 
Fix natural numbers $j_2,\dots,j_n$. Letting
\[\V_j= \overline{\ran\Phi_j^1\otimes\U_1+\sum_{i=2}^n
\ran\Phi_{j_i}^i\otimes\U_i}^{w^*}, \ \ \ j\in \bb{N},\] 
we see that $\cl V_{j+1}\subseteq \cl V_j$ for each $j$ and, by Theorem \ref{777}, $\sup_jk(\V_j)<\infty$.
By \cite[Corollary 4.21]{eletod2}, 
\[ \W_1 \deq\cap_{j\in\bN}\V_j=
\overline{\M_1\otimes\U_1+\sum_{i=2}^n
\ran\Phi_{j_i}^i\otimes\U_i}^{w^*} . \]
By Lemma \ref{1.6}, the space $\W_1$ is hyperreflexive. 
Continuing inductively, we see that the space
\[ \W_m \deq
\overline{ \sum_{i=1}^m \M_i\otimes\U_i +    
\sum_{i=m+1}^n \ran\Phi_{j_{i}}^{i}\otimes\U_{i}}^{w^*}  \] is hyperreflexive for each $m = 1,\dots,n$; 
in particular, the space $\W=\W_n$ is hyperreflexive. 

Let $\cl H$ be a separable Hilbert space. 
The space  $\W\bar\otimes \cl B(\cl H) $ is unitarily equivalent to
\[ \overline{ (\M_1 \bar\otimes \cl B(\cl H)) \otimes(\U_1 \bar\otimes \cl B(\cl H)) + \cdots +
(\M_n \bar\otimes \cl B(\cl H)) \otimes (\U_n \bar\otimes \cl B(\cl H))}^{w^*}.\]  

Since the spaces $\M_i \bar\otimes \cl B(\cl H)$ are ternary masa bimodules, while the spaces 
$\U_i \bar\otimes \cl B(\cl H)$ are completely hyperreflexive, 
by the first part of the proof, the space $\W\bar\otimes B(\cl H) $ is hyperreflexive.
\end{proof}

\begin{corollary}\label{c_tenvn}
If $\cl M$ is a von Neumann algebra of type I and $\cl A$ is a nest algebra then 
$\cl M\bar\otimes\cl A$ is hyperreflexive and $k(\cl M\bar\otimes\cl A) \leq 5$.
\end{corollary}
\begin{proof} 
Immediate from Theorem \ref{888} or, alternatively, from Corollary \ref{xxx4}. 
\end{proof}

\section*{Acknowledgements}

The authors are grateful to an anonymous referee for helpful
comments and suggestions.

\end{document}